\documentclass[a4paper,11pt]{amsart}%
\usepackage{graphicx}        % standard LaTeX graphics tool
\usepackage{multicol}        % used for the two-column index
\usepackage[justification=centering]{caption}
\usepackage{microtype}
\usepackage{algorithm}
\usepackage[noend]{algorithmic}
\usepackage{amssymb}
\usepackage{amsmath}
\usepackage{amsfonts}
\usepackage{booktabs}
\usepackage{multirow}
\usepackage{exscale}
\usepackage{latexsym}
\usepackage{xspace}
\usepackage{enumerate}
\usepackage{todonotes}
\usepackage{hyperref}
\usepackage{cleveref}

\usepackage{enumitem}
\newlist{cvdesc}{description}{1}
\setlist[cvdesc]{nosep,
labelindent=0pt,
labelwidth=2.8cm,
labelsep*=0.2cm,
leftmargin=cm,
font=\normalfont,
align=right}
\newlist{compactenum}{enumerate}{3}
\setlist[compactenum]{nosep,itemsep=2pt,topsep=2pt,labelindent=1em,leftmargin=1em}
\setlist[compactenum,1]{label=\textbullet}
\newtheorem{theorem}{Theorem}
\newtheorem{lemma}[theorem]{Lemma}
\newtheorem{proposition}[theorem]{Proposition}
\newtheorem{corollary}[theorem]{Corollary}
\newtheorem{remark}[theorem]{Remark}
\newtheorem{definition}[theorem]{Definition}

\newtheorem{example}[theorem]{Example}

\newcommand{\std}{{\ttfamily std}\xspace}
\newcommand{\slimgb}{{\ttfamily slimgb}\xspace}
\newcommand{\singular}{{\sc Singular}\xspace}
\newcommand{\plural}{{\sc Plural}\xspace}
\newcommand{\field}{\ensuremath{K}\xspace}
\newcommand{\fieldx}{\ensuremath{K}}
\newcommand{\PBW}{$G$}
\newcommand{\N}{\mathbb{N}}
\newcommand{\Q}{\mathbb{Q}}
\newcommand{\Z}{\mathbb{Z}}

\newcommand{\K}{\field}
\renewcommand{\d}{\partial}

\DeclareMathOperator{\Ann}{Ann}
\DeclareMathOperator{\cusp}{cusp}
\DeclareMathOperator{\Reiffen}{Reiffen}
\DeclareMathOperator{\LNF}{LNF}
\DeclareMathOperator{\LSP}{LSP}
\begin{document}

\title{Modular Techniques For Noncommutative Gr\"obner Bases}
%\title{Modular Techniques For Computing Non-Commutative Gr\"obner Bases}
%% Use \authorrunning{Short Title} for an abbreviated version of
%% your contribution title if the original one is too long
\author[W. Decker]{Wolfram Decker}
\address{Department of Mathematics, University of Kaiserslautern, Erwin-Schr\"odinger-Str., 67663 Kaiserslautern, Germany}
\email{decker@mathematik.uni-kl.de}
\author[C. Eder]{Christian Eder}
\address{Department of Mathematics, University of Kaiserslautern, Erwin-Schr\"odinger-Str., 67663 Kaiserslautern, Germany}
\email{ederc@mathematik.uni-kl.de}
\author[V. Levandovskyy]{Viktor Levandovskyy$^\dagger$}
%\author[V. Levandovskyy]{Viktor Levandovskyy}
\address{Lehrstuhl D f\"ur Mathematik, RWTH Aachen University, Templergraben 64, 52062 Aachen, Germany}
\email{levandov@math.rwth-aachen.de}
\author[S. K. Tiwari]{Sharwan K. Tiwari$^*$}
%Supported by the German Academic Exchange Service DAAD.
\address{Department of Mathematics, University of Kaiserslautern, Erwin-Schr\"odinger-Str., 67663 Kaiserslautern, Germany}
\email{stiwari@mathematik.uni-kl.de}
\thanks{$^\dagger$ Corresponding author} 
\thanks{$^*$ Supported by the German Academic Exchange Service (DAAD)}

\maketitle

\begin{abstract}
In this note, we extend modular techniques for computing Gr\"obner bases from the
commutative setting to the vast class of noncommutative $G$-algebras. As in
the commutative case, an effective verification test is only known to us in the 
graded case. In the general case, our algorithm is probabilistic in the sense 
that the resulting Gr\"obner basis can only be expected to  generate the given ideal,
with high probability. We have implemented our algorithm in the computer 
algebra system {\sc{Singular}} and give timings to compare its performance  
with that of other instances of Buchberger's algorithm, testing examples 
from $D$-module theory as well as classical benchmark examples.
A particular feature of the modular algorithm is that it allows parallel runs.
\end{abstract}

\section{Introduction}

That the concept of Gr\"obner bases and Buchberger's algorithm for computing these
bases can be extended from the commutative to the noncommutative setting 
was remarked in the late 1980s in the case of Weyl algebras ~\cite{Castro},~\cite{Takayama},
with particular emphasis on the computational treatment of $D$-modules.
\par
Around the same time, Apel~\cite{Apel} introduced a much more general class of algebras,
called  \textit{$G$-algebras}, which are well-suited for Gr\"obner basis methods 
(see ~\cite{Levandovskyy2005, VikHans} for some details). These algebras are 
defined over a field, and are also known as \textit{algebras of solvable type}~\cite{KW, HLi, Kr}, or as 
\textit{PBW-algebras}~\cite {BGV, Roman}. They include the  Weyl algebras together with a variety 
of other important algebras, such as universal enveloping algebras of finite dimensional Lie algebras, 
many quantum algebras (including coordinate rings of quantum affine planes and algebras of quantum 
matrices as well as some quantized enveloping algebras of Lie algebras), and numerous 
algebras formed by common linear partial functional operators. 
\par
The use of Gr\"obner bases allows, more generally, the computational treatment 
of $GR$-algebras, which are factor algebras of $G$-algebras by two-sided ideals, and 
which include Clifford algebras (in particular, exterior algebras) and a number of quantum 
algebras, such as quantum general and quantum special linear groups.
\par
Over the rationals, modular methods not only enable us to avoid intermediate coefficient swell,
but also provide a way of introducing parallelism into our computations.
In the context of Gr\"obner bases, this means to reduce the given ideal modulo several 
primes, compute a Gr\"obner basis for each reduced ideal, and use Chinese remaindering and 
rational reconstruction to find the desired Gr\"obner basis over $\mathbb Q$. 
See~\cite{A, elbert83, pfister2011} for the commutative case. 
\par
In this paper, we extend the modular Gr\"obner basis algorithm to \PBW-algebras. In a final verification step,
the algorithm checks that the result $\mathcal G$ is indeed a left or right Gr\"obner basis, and that the left or right
ideal generated by $\mathcal G$ contains the left or right ideal we started with. In the graded
case, this guarantees the equality of the two ideals and, thus, that $\mathcal G$ is a Gr\"obner 
basis for the given ideal. In the general case, we can only expect equality, with high 
probability. Alternatively, we may apply the algorithm in a homogenized situation, and 
dehomogenize the result (if this is computationally feasible).

\par
The paper is organized as follows:
 In Section~\ref{basicPBW}, we recall the definition and basic properties of $G$-algebras. Then 
 we address gradings, filtrations and the homogenization of \PBW-algebras, and
 Gr\"obner bases. In Section~\ref{sec:modgb}, we present our modular
 algorithm and discuss the final verification test for the graded case.
 In Section~\ref{compresult}, based on our implementation in the 
computer algebra system \singular~\cite{singular, plural}, we compare
the performance of the modular algorithm with that of other variants 
of Buchberger's algorithm. In Section \ref{sect:con},
we conclude the paper with final remarks.
 
\section{Preliminaries}\label{basicPBW}
In this section, we introduce some of the terminology used in this paper.

\subsection{Basic Notation}
\label{sec:1}
We work over a field $K$. Given a finite set of indeterminates
$x := \{x_1,\ldots,x_n \}$, we write ${\left\langle {x}\right\rangle}:=\langle x_{1},\ldots, x_{n} \rangle$ for the free \textit{monoid} on
${x}$, and consider the corresponding monoid algebra $\fieldx\langle {x}\rangle:= \fieldx\langle
x_{1},\ldots, x_{n} \rangle$, that is, the \textit{free associative \field-algebra} generated by $\langle x \rangle$.
\par   
A \textit{monomial} in $x_1,\ldots,x_n$ is an element of $\langle x \rangle$, that is, a word 
in the finite alphabet $x$. A \textit{standard monomial} is a monomial of type  
$x^\alpha=x_{1}^{\alpha_{1}}\cdots x_{n}^{\alpha_{n}}$, where $\alpha=(\alpha_1,\dots,\alpha_n)\in\N^n$. %; otherwise, we speak of a non-standard monomial.
  A \textit{standard term} in $\field\langle x \rangle$ is an
  element of $\field$ times a standard monomial.
  A \textit{standard polynomial} in $\field\langle x \rangle$ is a sum $f= \sum c_{\alpha}x^{\alpha}$
  of finitely many nonzero standard terms involving distinct standard monomials. 
  The  \emph{Newton diagram} of such an $f$ is 
 $$
\mathcal N(f)=\{\alpha\in\N^n \mid c_{\alpha}\neq 0\}.
$$
By convention, the zero element is a standard polynomial, with $\mathcal N(0)=\emptyset$.

 \par
To give a partial ordering $>$ on the set of standard monomials means to give a 
partial ordering $>$ on $\N^n$. We only consider orderings which are total, are such that 
$\alpha > \beta$ implies $\alpha + \gamma  > \beta + \gamma$, for all $\alpha, \beta, 
\gamma\in\N^n$, and are well-orderings. By abuse of language, 
we then say that $>$ is a \emph{monomial ordering on}
${\left\langle {x}\right\rangle}$. Given $>$, it makes sense to speak of the \emph{leading exponent} 
$\exp(f)=\exp_>(f)$ %and the \emph{leading monomial} $\LM(f) =\LM_>(f)$ 
of a standard polynomial $f\neq 0$.
The lexicographic and degree reverse lexicographic
orderings are defined as usual. Similarly for block orderings.
We write $\epsilon_i=\exp(x_i)=(0,\dots,0,1,0,\dots,0)$, for all $i$.

Given $\omega\in\N^n$, the $\omega$-\textit{weighted degree} of     
 a standard monomial $x^{\alpha}$ is 
 $$
\vert \alpha \vert_{\omega} = \langle \omega, \alpha \rangle = \omega_1 \alpha_1+\cdots+\omega_n \alpha_n,
 $$
while that of a standard polynomial $f\neq 0$
is 
$$
\deg_{\omega}(f)=\max\{\vert \alpha \vert_{\omega} \mid \alpha\in\mathcal N(f)\}.
$$
We say that $f$ is $\omega$-\emph{homogeneous} (of degree $\deg_{\omega}(f)$) 
if $\vert \alpha \vert_{\omega} = \deg_{\omega}(f)$ for all $\alpha \in \mathcal N(f)$.

By convention, the zero element is considered to be $\omega$-homogeneous of each 
degree $d\in\N$, and we write $\deg_{\omega}(0)=-\infty$. For any monomial 
ordering $>$ on ${\left\langle {x}\right\rangle}$,  we obtain a new monomial ordering 
$>_{\omega}$ on ${\left\langle {x}\right\rangle}$ by setting 
$$
\alpha>_{\omega} \beta%\Leftrightarrow
\Longleftrightarrow
\left\{
\begin{array}{lcl}
\vert \alpha \vert_{\omega}  >  \vert \beta \vert_{\omega} \\
\text{ or}\\
\vert \alpha \vert_{\omega}  =  \vert \beta \vert_{\omega}  \ \text{ and } \ \alpha >\beta\ .
\end{array}\right .
$$

\subsection{$G$-algebras}

Each  finitely presented associative $K$-algebra $A$ is isomorphic to a factor algebra 
of type $K\langle x_1, \ldots, x_n \rangle/J$, for some $n$ and some two-sided ideal 
$J \subset K\langle x_1, \ldots, x_n \rangle$. If $J$ is given by a set of two-sided
generators $g_1, \dots, g_r$, then we say that \emph{$A$ is generated by 
$x_1, \ldots, x_n$ subject to the relations $g_1, \dots, g_r$}, and write
$$
A = \K\langle x_1, \dots , x_n \mid g_k=0, \ 1\leq k \leq r \rangle.
$$

We say that $A$ as has a \emph{Poincar\'e-Birkhoff-Witt (PBW) basis} if the 
standard monomials in $\K\langle x_1, \dots , x_n\rangle$ represent a $\K$-basis for $A$.
In this case, every element $f\in A$ has a unique representation
$$
f=\sum_{\alpha\in\N^n}c_{\alpha}x^{\alpha},
$$
where we abuse our notation by denoting the class of the standard monomial
$x^{\alpha}$ in $A$ also by $x^{\alpha}$. We then refer to the $c_{\alpha}$ as the
\emph{coefficients} of $f$ and let  the \emph{Newton diagram} $\mathcal N(f)$, 
the $\omega$-\emph{weighted degree} $\deg_{\omega}(f)$, and the property of being
$\omega$-homogeneous be defined as above. Similarly for the \emph{leading exponent} 
$\exp(f)=\exp_>(f)$ %and the \emph{leading monomial} $\LM(f)=\LM_>(f)$ 
with respect to a monomial ordering $>$ on $\langle x\rangle$ if  $f\neq 0$. 

$G$-algebras are obtained by imposing specific commutation relations:% on $K\langle x \rangle$:

\begin{definition}
\label{Galg}
A \emph{$G$-algebra} over $K$ is a factor algebra of type
\[
A = \K\langle x_1, \dots , x_n \mid x_jx_i = c_{ij} \cdot x_ix_j+d_{ij},  \ 1\leq i<j \leq n \rangle,
\]
where the $c_{ij}\in K$ are nonzero scalars and the $d_{ij}\in \K\langle x_1, \dots , x_n\rangle$ are 
standard polynomials such that the following two conditions hold:  
%\begin{enumerate} 
\begin{compactenum} 
\item\label{cond-GB-1} There exists a monomial ordering $>$ on $\langle x_1, \dots , x_n\rangle$ such that
$$
d_{ij}=0\ \text{ or }\ \epsilon_i+\epsilon_j>\exp(d_{ij})\ \text{ for all }\ 1\leq i<j \leq n.
$$
Every such ordering is called \emph{admissible} for $A$.
\item  For all $\ 1 \leq i<j<k \leq n$, the elements
$$
c_{ik}c_{jk} \cdot d_{ij}x_k - x_k d_{ij} + c_{jk} \cdot x_j d_{ik} - c_{ij} \cdot d_{ik}x_j + d_{jk}x_i - c_{ij}c_{ik} \cdot x_i d_{jk}
$$
reduce to zero with respect to the relations of $A$.
\end{compactenum} 
%\end{enumerate}
\end{definition}

\begin{example}
\label{ex1}
The \emph{$n$th Weyl algebra over $K$} is the $G$-algebra 
$$
D_n(\field)=\field\langle x_1,\ldots, x_n,
\partial _1,\ldots \partial _n \mid \partial_i x_i=x_i\partial _i +1, \partial _i x_j=x_j \partial _i \ \text { for }\ i\neq j\rangle, 
$$
where we tacitly assume that $x_j x_i=x_i x _j$ and $\partial _j \partial_i=\partial_i \partial _j$ for all $i,j$.
Note that any  monomial ordering on $\langle x, \partial\rangle$ is admissible for $D_n(\field)$. 
\end{example}

\begin{example}  
\label{ex:ex2}
The \emph{$r$th shift  algebra over $K$} is the $G$-algebra 
$$
S_r(K)=\langle s_1, \dots, s_r,  t_1, \dots, t_r \mid t_j s_k = s_k t_j - \delta_{jk} t_j  \rangle,
$$
where we tacitly assume that $s_j s_i=s_i s_j$ and $t_j t_i=t_i t_j$ for all $i,j$. Note that
any monomial ordering on $\langle s,t\rangle$ is admissible for $S_r(K)$. 
\end{example}

\begin{example}  
\label{ex:ex-tensor}
If $A =K\langle x_1,\dots, x_n\mid C_A\rangle$ and $B=K\langle y_1,\dots, y_m \mid C_B\rangle$ 
are $G$-algebras, then  their tensor product $A\otimes_K B$ is a $G$-algebra as well,
$$
A\otimes_K B=K\langle x_1,\dots, x_n,   y_1,\dots, y_m \mid C_A, C_B, y_jx_i=x_iy_j\rangle.
$$
Note that if $>_A$ and $>_B$ are admissible orderings for $A$ and $B$, respectively, then
the block ordering $(>_A,>_B)$ is admissible for $A\otimes_K B$.
\end{example}

$G$-algebras enjoy structural properties which are reminiscent of those
of commutative polynomial rings (see \cite{BGL01, KL, vikThesis, Levandovskyy2005} for definitions 
and proofs\footnote{In the context of this paper, we should point out that the Noetherian property follows 
as in the commutative case from Dickson's lemma using Gr\"obner bases. See Section \ref{sec:GB}.}):

\begin{proposition}%[Properties of $G$-algebras]
Let $A = \K\langle x_1, \dots , x_n \mid C \rangle$ be a $G$-algebra.  Then: 
\begin{itemize}
%\begin{compactenum}
\item $A$ has a PBW-basis;
\item $A$ is left and right Noetherian domain;
%\item $A$ is an integral domain;
\item the Gel'fand-Kirillov dimension of $A$ over $\K$ is equal to $n$;
\item the global homological dimension of $A$ is at most $n$;
\item the generalized Krull dimension  of $A$ is at most $n$;
\item $A$ is Auslander regular and Cohen-Macaulay.
%\end{compactenum}
\end{itemize}
\end{proposition}
\noindent

Computing in a $G$-algebra rather than in a commutative polynomial ring means to
additionally apply the commutation relations in the definition of the $G$-algebra. Over
the rationals, this typically leads to even more coefficient swell. The following example 
illustrates this point and indicates, thus, the particular importance of modular methods
for computations in $G$-algebras over $\Q$. With regard to notation, $[a+b]^n$ stands 
for writing out the right hand side of the binomial formula in its commutative version.  

\begin{example}[\cite{LKM11}] Suppose that $K$ is a field of characteristic zero. %\todo{latex: compactenumerate verbessern}
\begin{compactenum}
\item In the Weyl algebra $D_1(K)=\K\langle x, \d \mid \d x = x \d + 1 \rangle$, we have
$$(x+\d)^n=[x+\d]^n+\sum_{k=0}^{n-2}\sum_{j=0}^{n-k-2}{n\choose j}{n-j\choose k}g(n-j-k)x^k \d^j,$$
where $g(n)=(n-1)!!$ if $n$ is even, and $g(n)= 0$ otherwise. \\
\item In the shift algebra $S_{1}(\K)=K\langle s, t \mid ts=st-t \rangle$, we have
$$(s+t)^n=[s+t]^n+\sum_{k=0}^{n-1}\sum_{j=0}^{n-k-1}(-1)^{n+k+j}{n\choose k}S(n-k,j)s^kt^j,$$
where the $S(n,k)$ denote the Stirling numbers of the second kind.
\end{compactenum}
\end{example}

In what follows, we will summarize some results on $G$-algebras, left or right ideals
in $G$-algebras, and left or right Gr\"obner bases for these ideals. For simplicity of the
presentation, we will focus on the case of left ideals. If $T\subset A$ is any subset of the 
$G$-algebra $A$, the notation $\langle T \rangle={}_A\langle T \rangle$ will always refer 
to the left ideal of $A$ generated by $T$.

% descr of ords
\subsection{Graded $G$-algebras}
\label{subsec: gradings}

In this subsection, we consider a $G$-algebra
\[
A = \K\langle x_1, \dots , x_n \mid x_jx_i = c_{ij} \cdot x_ix_j+d_{ij},  \ 1\leq i<j \leq n \rangle
\]
such that the $c_{ij} \cdot x_ix_j+d_{ij}$ are $\omega$-homogeneous for some
weight vector $0\neq\omega\in\N^n$. That is, 
\begin{equation}
\label{equ:graded}
\omega_i + \omega_j=\vert \alpha \vert_{\omega} \ \! \text{ for } \ 1\leq i < j \leq n \ \text{ and all }  \ \alpha \in \mathcal N(d_{ij}).
\end{equation}
Then $A$ is \emph{graded} with respect to the $\omega$-weighted degree,

\begin{align*}
A = \bigoplus_{d\geq 0} A_d, & {} \text{ with }  
\ A_d = \left\{f \in A \mid f \ \omega\text{-homogeneous for all } \alpha\right\},\\[-4\jot]
& {} \text{ and where } \ A_dA_e\subset A_{d+e}  \text{ for all } d,e
\end{align*}
(see \cite[Chapter 4, Section 6]{BGV} for a detailed discussion). In this case, a left
ideal $I\subset A$ is \emph{graded} if it inherits the grading:
$$
I = \bigoplus_{d\geq 0} I_d=\bigoplus_{d\geq 0} (A_d\cap I)\subset A.
$$ 
Equivalently, $I$ is generated by (finitely many) $\omega$-homogeneous elements. 

If $\omega \in\N_{\geq 1}^n$, the $A_d$ are finite-dimensional $K$-vector spaces,
with $A_0=K$. We may then talk about the \emph{Hilbert function}
$$
H_{I}: \N \rightarrow \N, \ d \mapsto \dim_K I_d
$$
of every graded left ideal  $I\subset A$.

\subsection{Filtrations and Homogenization}
\label{subsec: homogenization}

Given any $G$-algebra
\[
A = \K\langle x_1, \dots , x_n \mid x_jx_i = c_{ij} \cdot x_ix_j+d_{ij},  \ 1\leq i<j \leq n \rangle,
\]
a weight vector as in the previous subsection may not exist (consider the Weyl and shift  
algebras). In contrast, there is always an $\omega \in\N_{\geq 1}^n$  satisfying
\begin{equation}
\label{equ:filtered}
\omega_i + \omega_j>\vert \alpha \vert_{\omega} \ \! \text{ for } \ 1\leq i < j \leq n \ \text{ and all }  \ \alpha \in \mathcal N(d_{ij})
\end{equation}
(see \cite[Chapter 3, Section 1]{BGV}). In practical terms, such an $\omega$ can be found 
by solving the following linear programming problem: 

${\it minimize} \ \displaystyle \sum_{i=1}^n \omega_i$ \ \textit{ subject to }\\

\begin{itemize}
\item $\omega_i + \omega_j > \vert \alpha \vert_{\omega} 
\ \! \text{ for } \ 1\leq i < j \leq n \ \text{ and all }  \ \alpha \in \mathcal N(d_{ij}),$
\item  $\omega_1, \ldots, \omega_n > 0$.\\
\end{itemize}

\begin{example}
\label{ex:Weyl2}
For both the Weyl algebra $D_n(K)$ and the shift algebra $S_r(K)$, any $\omega \in \N_{\geq 1}^n$
will do.
\end{example}

\noindent
Suppose now that an $\omega\in\N_{\geq 1}^n$ satisfying \eqref{equ:filtered} is given, and fix a monomial 
ordering $>$ on $\langle x\rangle$. Then the induced ordering $>_\omega$ is admissible 
for $A$. Furthermore, we get a filtration of $A$, the $\omega$-\emph{filtration}, if we set
$$
F_d^{\;\!\omega} \!A =\{f\in A \mid \deg_{\omega}(f)\leq d\} \ \text{ for all } \ d\in\N.
$$
The \emph{Rees algebra} corresponding to this filtration is the subalgebra
$$
R^{\;\!\omega} \!A=\bigoplus_{d\geq 0} (F_d^{\;\!\omega} \!A) t^d\subset A[t]=A\otimes_K K[t].
$$
Note that each element $F\in (F_d^{\;\!\omega} \!A) t^d$ has the form
$$
F=\sum_{\vert \alpha \vert_{\omega}\leq d} c_{\alpha}{x}^{\alpha} t^d=
\sum_{\vert \alpha \vert_{\omega}\leq d} c_{\alpha}\tilde{x}^{\alpha} t^{d-\vert \alpha \vert_{\omega}},
$$
where
$$
\tilde{x}_i=x_it^{\omega_i} \ \text{ and } \ \tilde{x}^{\alpha} = 
\tilde{x}_1^{\alpha_1}\cdots \tilde{x}_n^{\alpha_n}.
$$
So $F$ is $\widetilde{\omega}$-homogeneous of degree $d$, where
$\widetilde{\omega}=(\omega, 1)\in\N^{n+1}$.
In particular, if we set
$$
\widetilde{d}_{ij}={d}_{ij}t^{\omega _i +\omega _j}=\sum_{\alpha\in\mathcal N({d}_{ij})} 
c^{(ij)}_{\alpha} x^\alpha t^{\omega _i +\omega _j} =\sum_{\alpha\in\mathcal N({d}_{ij})} c^{(ij)}_{\alpha}
\tilde{x}^{\alpha}t^{\omega _i +\omega _j-|\alpha|_{{\omega}}},
$$
then the $c_{ij} \cdot \tilde{x_i}\tilde{x_j} + \widetilde{d}_{ij}$ are $\widetilde{\omega}$-homogeneous
of degree $\omega _i +\omega _j$.
We introduce a  monomial ordering $>_{{\omega}}^h$ on $\langle \tilde{x},t\rangle$ by setting 
\[
(\alpha, d) >_{{\omega}}^h (\beta, e)
\Longleftrightarrow
\left\{
\begin{array}{lcl}
|(\alpha,d)|_{\widetilde{\omega}}  >  |(\beta, e)|_{\widetilde{\omega}}\\
\text{ or}\\
|(\alpha,d)|_{\widetilde{\omega}}  =  |(\beta, e)|_{\widetilde{\omega}}\ \text{ and } \ \alpha >_{{\omega}}\beta\ .
\end{array}\right .
\]

\noindent
Then we have:
\begin{theorem}%[The Rees Algebra as a $G$-Algebra]
\label{ReesConstruction}
In the situation above, the Rees algebra $R^{\;\!\omega} \! A$ is a graded $G$-algebra,
$$
R^{\;\!\omega} \! A=\bigoplus_{d\geq 0} (F_d^{\;\!\omega} \! A) t^d= \K\langle \tilde{x}_1, \dots , \tilde{x}_n, t \mid \widetilde{C} \rangle,
$$
with commutation relations
$$
\widetilde{C}: \tilde{x_j}\tilde{x_i} = c_{ij} \cdot \tilde{x_i}\tilde{x_j}
+ \widetilde{d}_{ij},  \ t \tilde{x_i} = \tilde{x_i} t,
$$
and admissible ordering $>_{{\omega}}^h$.  %$\hat{d_{ij}}(\tilde x,t)$, where $\hat{d_{ij}}(\tilde x,t)$ is a polynomial
%$t^{\omega_i + \omega_j} d_{ij}(x)$, rewritten as a polynomial in variables $\tilde x,t$.
\end{theorem}
\begin{proof}
Clear from the discussion above.  
%In the second conditions, there are three binomials of a similar
%construction. Consider the first such binomial for the Rees algebra
%$\widetilde{b}_1:=c_{ik}c_{jk} \cdot \widetilde{d}_{ij} \widetilde{x}_k - \widetilde{x}_k \widetilde{d}_{ij}$.
%By using the definitions
%$\widetilde{d}_{ij}={d}_{ij}t^{\omega_i +\omega_j}$, $
%\widetilde{x}_k=x_k t^{\omega_k}$,
%we rewrite the first summand of $\widetilde{b}_1$ into
%$c_{ik}c_{jk} \cdot d_{ij} x_k t^{\omega_i +\omega_j + \omega_k}$
%and similarly, the second summand into
%$x_k d_{ij} t^{\omega_i +\omega_j + \omega_k}$, thus
%$\widetilde{b}_1 = b_1 t^{\omega_i +\omega_j + \omega_k}$ for $b_1$
%from $A$. Analogous rewriting shows, that
%$\widetilde{b}_1 + \widetilde{b}_2 + \widetilde{b}_3 =
%(b_1 + b_2 + b_3)t^{\omega_i +\omega_j + \omega_k}$ and the claim
%follows, since $b_1 + b_2 + b_3$ rewrites to zero with the relations of $A$.
\end{proof}
\begin{definition}[Homogenization and Dehomogenization] In the situation above, the 
\emph{homo\-ge\-ni\-za\-tion} of an element $f\in A$ is the 
${\widetilde{\omega}}$-homogeneous element $f^h=ft^d\in R^{\;\!\omega} \! A$,
where $d=\deg_\omega(f)$. 
The \emph{homogeni\-za\-tion} of a left ideal $I\subset A$ is the graded left ideal 
$$
I^h=\langle f^h \mid f\in I\rangle = \bigoplus_{d\geq 0} (F_d^{\;\!\omega} \! A\cap I) t^d\subset R^{\;\!\omega} \! A.
$$
The \emph{dehomogenization} of an element $F\in (F_d^{\;\!\omega} \! A) t^d$ is the element $F\mid_{t=1}\in A$.
\end{definition}

\noindent 
Note that if $I$ is given by generators $f_1,\dots, f_r$, then
\begin{equation}
\label{equ:homog}
I^h = \langle f_1^h,\dots, f_r^h\rangle : t^\infty. 
\end{equation}

\noindent

\subsection{Gr\"obner Bases in \PBW-Algebras}  
\label{sec:GB}
The concept of  Gr\"obner bases extends from commutative polynomial rings
to \PBW-algebras. We give a brief account of this, referring to \cite{BGV,  MR2503959,  
vikThesis} for details and proofs.

To begin with, recall that a nonempty subset $E\subset \N^n$
is called a \emph{monoideal} if $E +\N^n =E$. Dickson's 
lemma tells us that each such $E$ is \emph{finitely generated}: There exist 
$\alpha_1,\dots, \alpha_s\in \N^n$ such that $E=\bigcup_{i=1}^{s} 
\left({\alpha_i} + \N^n\right)$.

Given a $G$-algebra 
\[
A = \K\langle x_1, \dots , x_n \mid x_jx_i = c_{ij} \cdot x_ix_j+d_{ij},  \ 1\leq i<j \leq n \rangle,
\]
an admissible ordering $>$ for $A$, and a subset $I\subset A$, we set
$$
\exp(I)=\{\exp(f)\mid f\in I\setminus \{0\}\} .
$$
Note that if $I$ is a nonzero left ideal of $A$, then $\exp(I)$ is a monoideal of  $\N^n$.  
Moreover, Macaulay's classical result on factor rings of commutative 
polynomial rings extends as follows:

\begin{remark}
\label{rem:Macaulay}
Let $I$ be a nonzero left ideal of $A$. Then 
the standard monomials $x^\alpha, \ \alpha\in \N^n\setminus\exp(I)$, represent 
a $K$-vector space basis for $A/I$ (see, for example, \cite[Proposition 9.1]{mora1}).
\end{remark}

\begin{definition} Let $I$ be a  nonzero left ideal of $A$. 
A \emph{(left) Gr\"obner basis} for $I$ (with respect to $>$) is a finite subset  
$\mathcal G=\{g_{1},\ldots,g_{s}\} \subset I\setminus \{0\}$ such that
$$
\exp(I)=\bigcup_{i=1}^{s} \left(\exp(g_i) + \N^n\right).
$$
A finite subset  $\mathcal G\subset A\setminus \{0\}$ is  a \emph{left Gr\"obner basis} 
if it is a Gr\"obner basis for the left ideal $\langle \mathcal G\rangle\subset A$ it generates.
\end{definition} 

As in the commutative case, every nonzero left ideal $I\subset A$ has a Gr\"obner 
basis, and every such basis generates $I$ as a left ideal. Furthermore, we have the
concepts of left normal forms ($\LNF$) (that is,  left division with remainder)
and left S-polynomials ($\LSP$) in $G$-algebras. Based on this, there are
adapted versions of Buchberger's criterion and, thus,  Buchberger's algorithm
which allow us to characterize and compute left Gr\"obner bases, respectively:
\begin{theorem}[Buchberger's Criterion]\label{BuchCrit} Let $\mathcal G=\{g_{1},\ldots,g_{s}\}$ be a finite 
subset of $ A\setminus \{0\}$. Then $\mathcal G$ is a left Gr\"obner basis 
iff 
$$\LNF(\LSP(g_i,g_j), \mathcal G) = 0 \;\text{ for }\; 1\leq i<j \leq s.
$$
\end{theorem}
%
%\begin{proof} See, for example, \cite[Chapter 2, Theorem 6.5]{BGV}.
%\end{proof}
The following result will be useful for establishing our modular algorithm.
\begin{corollary}[Finite Determinacy of Gr\"obner Bases]\label{Finite_det} 
Let $I$ be a nonzero left ideal of $A$, and let $\mathcal G=\{g_{1},\ldots,g_{s}\}$ be a 
Gr\"obner basis for $I$ with respect to $>$. There exists a finite set $F$ of standard 
monomials such that if   $>_1$ is any admissible ordering for $A$ which coincides with 
$>$ on $F$, then:
\begin{enumerate}
\item $\exp_{>}(g_i)=\exp_{>_1}(g_i)$ for all $g_i \in \mathcal G$.
\item $\mathcal G$ is a Gr\"obner basis for $I$ also with respect to $>_1$. 
\end{enumerate}
\end{corollary}
\begin{proof}
Let $F$ be the set of all standard monomials of all elements of $A$ occurring
during the reduction process of the $\LSP(g_i,g_j)$ to zero modulo $\mathcal G$ 
with respect to $>$. Then for any admissible ordering $>_1$ for $A$ which coincides 
with $>$ on $F$, the  $\LSP(g_i,g_j)$ also reduce to zero modulo $\mathcal G$ with 
respect to $>_1$. The result follows from Theorem \ref{BuchCrit}.
\end{proof}

The notion of a reduced left Gr\"obner basis is analogous to that in the commutative case, 
every left ideal of $A$ has a uniquely determined such basis, and this basis can 
be computed by a variant of Buchberger's algorithm. 

If $A$ is graded with respect to some weight vector $\omega$ as in Subsection
\ref{subsec: gradings}, and $>$ is an admissible ordering for $A$, then the 
induced ordering $>_\omega$ is also admissible for $A$. When computing
a Gr\"obner basis with respect to $>_\omega$, starting  from 
$\omega$-homogeneous elements, Buchberger' s algorithm will return 
Gr\"obner basis elements which are $\omega$-homogeneous as well.
In particular, reduced Gr\"obner bases consist of $\omega$-homogeneous
elements.

With regard to homogenizing and dehomogenizing Gr\"obner bases, 
the proposition known from the commutative case extends as follows:

\begin{proposition}\label{GBHomogenization} In the situation of 
\Cref{ReesConstruction},  let $I\subset A$ be a nonzero left ideal, and let 
$I^h \subset R^{\;\!\omega} \! A$ be its homo\-ge\-ni\-zation.  Then:
\begin{enumerate}
\item If $\mathcal G=\{g_1,\dots, g_s\}$ is a  Gr\"obner basis for $I$ with respect 
to  $>_{\omega}$, then $\mathcal G^h=\{g_1^h,\dots, g_s^h\}$ is a Gr\"obner basis 
for ${I^h}$ with respect to $>_{{\omega}}^h$\;  which consists of \;\!
${\widetilde{\omega}}$-homogeneous elements. 
\item Conversely, if \ $\mathcal G=\{G_1,\dots, G_s\}$ is a  Gr\"obner basis for $I^h$ 
with respect  to $>_{{\omega}}^h$\; which consists of \;\! ${\widetilde{\omega}}$-homogeneous 
elements, then \ \mbox{$\mathcal G\mid_{t=1}=\{G_1\mid_{t=1},\dots, G_s\mid_{t=1}\}$} is a 
Gr\"obner basis for $I$ with respect to  $>_{\omega}$.
\end{enumerate}
\end{proposition}

Part (1) of the proposition will be of theoretical use for establishing our 
modular algorithm. From a practical point of view, as already pointed out, we may 
wish to verify the correctness of the Gr\"obner basis returned by the modular 
algorithm by working in a homogenized situation (provided this is computationally
feasible). One possible approach for this is to compute $I^h$ using formula
\eqref{equ:homog}, and apply part (2) of the proposition. We describe a second 
approach which does not require to compute the saturation and is more flexible
with regard to monomial orderings.
%and is more flexible with regard to the choice of monomial ordering. 
For this, given $>$ and a vector $\omega\in\N^n_{\geq 1}$ 
as in Subsection \ref{subsec: homogenization}, consider the monomial ordering 
$>^h$ on $\langle \tilde{x},t\rangle$ defined by 
\[
(\alpha, d) >^h (\beta, e)
\Longleftrightarrow
\left\{
\begin{array}{lcl}
|(\alpha,d)|_{\widetilde{\omega}}  >  |(\beta, e)|_{\widetilde{\omega}}\\
\text{ or}\\
|(\alpha,d)|_{\widetilde{\omega}}  =  |(\beta, e)|_{\widetilde{\omega}}\ \text{ and } \ \alpha > \beta\ .
\end{array}\right .
\]
This ordering is admissible for $R^{\;\!\omega} \! A$ and we have:
\begin{proposition} 
In the situation above, let  $I=\langle f_1,\ldots,f_r\rangle \subset A$ 
be a nonzero left ideal, and let $J=\langle f_1^h,\ldots,f_r^h\rangle \subset R^{\;\!\omega} \! A$. 
If $\mathcal{G}$ is a Gr\"obner basis for $J$ with respect to $>^h$\; which consists of \
$\widetilde{\omega}$-homogeneous elements, then $\mathcal{G}\mid_{t=1}$ is a 
Gr\"obner basis for $I$ with respect to $>$.   
\end{proposition}
\begin{proof} Since dehomogenizing $J$ gives us back $I$, we have $\mathcal{G}\mid_{t=1}\subset I$. 

Let $f\in I$ be any nonzero element. Then $f^h\in I^h$, and we conclude from formula \eqref{equ:homog}
that there is an integer $e\in \N$ such that $f^h t^e \in J$. Since $\mathcal{G}$ is a Gr\"obner basis 
for $J$  with respect to $>^h$, there is an element $G\in \mathcal{G}$ such that $\exp_{>^h}(f^h t^e)=
\exp_{>^h}(G)+({\beta}, \widetilde{e})$ for some ${\beta}\in \N^{n}$ and some $\widetilde{e}\in\N$.
Then $\exp_{>}(f)=\exp_{>}(G\mid_{t=1})+\beta$,
which proves the proposition.
\end{proof}

In Subsection \ref{subsect:BSP}, we will consider particular instances of computing 
Bernstein-Sato polynomials to compare the performance of our modular 
algorithm with that of other versions of Buchberger's algorithm. Such computations
require the elimination of variables.
\begin{definition} Fix a subset $\sigma\subset \{1,\dots, n\}$, write $x_\sigma$ 
for the set of variables $x_i$ with $i\in\sigma$, and let $A_\sigma$ be the $K$-linear 
subspace of $A$ which is generated by the standard monomials in $\langle x_\sigma \rangle$. 
An \emph{elimination ordering}\ for $x\smallsetminus x_\sigma$ is an admissible ordering for 
$A$ such that
$$
f \in A\setminus \{0\},\ x^{\exp(f)}\in \langle x_\sigma \rangle  \ \text{ implies } \ f\in A_\sigma.
$$
\end{definition} 
\noindent
Suppose now that an elimination ordering $>$ as above exists. Then
$d_{ij}\in A_\sigma$ for each pair of indices $1\leq i<j \leq n$ with $i,j\in\sigma$. 
Furthermore, $A_\sigma$ is a subalgebra of $A$ with admissible ordering 
$>_\sigma$, where $>_\sigma$ is the restriction of $>$ to the set of standard monomials in 
$\langle x_\sigma\rangle$. Finally, if $I\subset A$ is a nonzero left ideal, and $\mathcal G$ 
is a Gr\"obner basis for $I$ with respect to $>$, then $\mathcal G\cap A_\sigma$ is a  
Gr\"obner basis for $I\cap A_\sigma$ with respect to $>_\sigma$. 

Note that in general, an elimination ordering for $x\smallsetminus x_\sigma$ may not exist. 
In practical terms, the question is whether the following linear programming
problem has a solution (see \cite{MR2503959}): 

${\it minimize} \ \displaystyle \sum_{i=1}^n \omega_i$ \ \textit{ subject to }\\

\begin{itemize}
\item $\omega_i + \omega_j \geq \vert \alpha \vert_{\omega}   
\ \! \text{ for } \ 1\leq i < j \leq n \ \text{ and all }  \ \alpha \in \mathcal N(d_{ij}),$
\item  $\omega_i = 0 \ \textit{ for } \ i\in\sigma\ \text{ and }\ \omega_i > 0 \ 
\textit{ for } \ i\in \{1,\dots, n\}\setminus \sigma$.\\
\end{itemize}

\noindent
If there is a solution $\omega$, and $>$ is any admissible ordering for $A$, then
$>_\omega$ is an elimination ordering for $x\smallsetminus x_\sigma$.
For computing Bernstein-Sato polynomials as discussed in Subsection 
\ref{subsect:BSP}, appropriate block orderings will do.

\begin{remark} 
\label{rem:right-two-sided}
The definition of a \emph{right Gr\"obner basis} is completely
analogous to that of left Gr\"obner basis, while a \emph{two-sided Gr\"obner
basis} is a left Gr\"obner basis $\mathcal G$ satisfying ${}_A\langle \mathcal G\rangle
={}_A\langle \mathcal G\rangle_A$. Having implemented means for
computing left Gr\"obner bases, right and two-sided Gr\"obner bases
are obtained by computing left Gr\"obner bases in the opposite algebra
$A^{\emph{opp}}$ and the enveloping algebra $A^{\emph{env}} =A\otimes_KA^{\emph{opp}}$, 
respectively. See \cite{vikThesis, Roman} for details.
\end{remark}
 
Rather than restricting ourselves just to $G$-algebras, we should finally point out that
the use of Gr\"obner bases as discussed above allows for an effective computationally 
treatment of a more general class of algebras:

\begin{remark}[$GR$-Algebras] A \emph{$GR$-algebra} is the quotient $A/J$ of a $G$-algebra 
$A$ by a two-sided ideal $J \subset A$. Having implemented $A$ in a computer algebra system such as 
{\sc{Singular}}, we can implement $A/J$ by computing a \emph{two-sided Gr\"obner basis} for $J$.
\end{remark}

\section{A Modular Gr\"obner Basis Algorithm for \PBW-Algebras}
\label{sec:modgb}

In this section, we extend the modular Gr\"obner bases algorithm from
commutative polynomial rings~\cite{A, elbert83, pfister2011}
to $G$-algebras. As before, we focus on left Gr\"obner bases.
By Remark \ref{rem:right-two-sided}, however, the algorithm 
presented below also gives modular ways of computing right 
and two-sided Gr\"obner bases of ideals (and modules).

Fix a $G$-algebra $A=\Q\langle x_1,\dots, x_n \mid C\rangle$ whose 
commutation relations 
$$
C: \ x_jx_i = c_{ij} \cdot x_ix_j+d_{ij},  \ 1\leq i<j \leq n,
$$ 
involve integer coefficients only, and
a monomial ordering on $\langle x \rangle$ which is admissible for $A$. 
Write $A_0$ for the subring of $A$ formed by the elements with integer 
coefficients. That is, $A_0$ is obtained from  the free associative $\mathbb{Z}$-algebra
$\mathbb{Z}\langle x_{1},\ldots,x_{n}\rangle$ by imposing the 
commutation relations $C$:
$$
A_0=\mathbb{Z}\langle x_{1},\ldots,x_{n} \mid C\rangle.
$$
 Similarly, if $N\geq 2$ is an integer which 
\begin{equation}
\label{cond:star}
\text{does neither divide any }  c_{ij} \text{ nor any coefficient of any }  d_{ij},
\end{equation}
%does not divide any $c_{ij}$, 
then write $A_N=\mathbb{Z}/N\mathbb{Z}\langle x_{1},\ldots,x_{n}
\mid C_N\rangle$, where $C_N$ is obtained from $C$ by reducing 
the $c_{ij}$ and the  coefficients of the $d_{ij}$ modulo $N$. Note that 
if $p$ is a prime satisfying \eqref{cond:star}, then $A_p$ is a 
$G$-algebra over the finite field $\mathbb F_p$.

If $\frac{a}{b}\in\mathbb{Q}$ with $\gcd(a,b)=1$ and $\gcd(b,N)=1$, set
$$\left(  \frac{a}{b}\right)  _{N}:=(a+N\mathbb{Z})(b+N\mathbb{Z})^{-1}%
\in\mathbb{Z}/N\mathbb{Z}.$$
 If $f\in A$ is an element such that $N$ is
coprime to the denominator of any coefficient of $f$, then its \emph{reduction
modulo $N$} is the element $f_{N}\in A_{N}$ obtained by mapping each
coefficient $c$ of $f$ to $c_{N}$. If $\mathcal H=\{h_{1},\dots,h_{t}\}\subset A$ is a
set of elements such that $N$ is coprime to the denominator of any coefficient 
of a coefficient of any $h_{i}$, set $\mathcal H_{N}=\{(h_{1})_{N},\dots,(h_{t})_{N}\}\subset A_N$. 

Let $I\subset A$ be a nonzero left ideal. We will explain how to compute a left
Gr\"obner basis for $I$ using modular methods. For this, we write
\[
I_{0}=I\cap A_0\;\text{ and }\;I_{N}=\left\langle f_{N}\mid f\in
I_{0}\right\rangle \subset A_{N}\text{,}%
\]
and call $I_{N}$ \emph{the reduction of $I$ modulo $N$}. 
We will rely on the following result:
\begin{lemma}
\label{lemma:realiting-ideals mod-p}
With notation as above, fix a set of generators $f_1, \dots, f_r$ for $I$ with coefficients 
in $\mathbb Z$. Then for all but finitely many primes $p$, the reduction $I_p$ is generated 
by the reductions of the $f_j$. That is, the ideal
$$
\widetilde{I}_p ={}_{A_p}\langle (f_1)_p, \dots, (f_r)_p \rangle
$$
coincides with $I_p$ for all but finitely many primes $p$.

\end{lemma}
\begin{proof}  
Let $g_1,\dots, g_s$ be a set of generators for the left ideal $I_0=I\cap A_0$.
Then each $g_i$ has a representation of type $g_i =\sum_{j=1}^r c_{ij} f_j $, with 
elements $c_{ij}\in A$. Clearing denominators in the coefficients
of the $c_{ij}$, we get a non-zero integer $d$ such that 
$d\cdot g_i \in {}_{A_0}\langle f_1, \dots, f_r \rangle$ for all $i$.
That is, 
$$
I_0={}_{A_0}\langle f_1, \dots, f_r \rangle : d.
$$
Then $I_p=\widetilde{I}_p$ for each prime $p$ which does not divide $d$.
The result follows.
\end{proof}

\begin{remark}
\label{rem:reduction-practical} When running our modular algorithm,
we will fix a set of generators $f_{1}, \ldots,f_{r}$ for $I$ as in
Lemma \ref{lemma:realiting-ideals mod-p}, and reject a prime $p$ if it does
not fulfil condition \eqref{cond:star}. 
If condition \eqref{cond:star} is fulfilled, we work with the left ideal $\widetilde{I}_{p}$
rather than with $I_p$. The  finitely many primes $p$ for which $\widetilde{I}_{p}$ and
$I_p$ differ will not influence  the final result if we use error tolerant rational reconstruction 
as introduced in~\cite{bdfp} and discussed below. 
\end{remark}
\noindent
In the following discussion, for simplicity of the presentation, we will ignore the  primes 
which do not fulfil condition \eqref{cond:star}.  
We will write $\mathcal G(0)$ for the reduced Gr\"{o}bner basis of $I$, and $\mathcal G(p)$ for that 
of $\tilde{I}_p$. The basic idea of the modular Gr\"obner basis algorithm is then as
follows: First, choose a finite set of primes $\mathcal{P}$ and 
compute $\mathcal G(p)$ for each $p\in\mathcal{P}$. Second, lift the $\mathcal G(p)$ 
coefficientwise to a set of elements $\mathcal G\subset A$. We then expect that $\mathcal G$
is a Gr\"{o}bner basis which coincides with our target Gr\"{o}bner basis
$\mathcal G(0)$.

The lifting process consists of two steps. First, use Chinese remaindering to
lift the $\mathcal G(p)\subset A_{p}$ to a set of elements $\mathcal G(N)\subset A_{N}$, 
with $N:=\prod_{p\in\mathcal{P}}p$\ . Second, compute a set of elements
$\mathcal G\subset A$ by lifting the coefficients occurring in $\mathcal G(N)$ 
to rational coefficients. Here, to identify Gr\"obner basis elements corresponding 
to each other, we require that $\exp(\mathcal G(p)) =
\exp(\mathcal G(q))$ for all $p,q\in\mathcal{P}$. 
This leads to the second condition in the definition below:

\begin{definition}\label{luckyP}%\cite{bdfp}
With notation as above, a prime $p$ is called \emph{lucky} if
\begin{enumerate}
\item [\emph{(L1)}] $I_{p}=\tilde{I}_{p}$ and
\item [\emph{(L2)}]$\exp(\mathcal G(0))=\exp(\mathcal G(p))$.
\end{enumerate}
\noindent
Otherwise, $p$ is called \emph{unlucky}.
\end{definition}   

\begin{lemma}
\label{lem finite} The set of unlucky primes is finite.
\end{lemma}

\begin{proof}
By Lemma \ref{lemma:realiting-ideals mod-p}, $I_{p}=\tilde{I}_{p}$ for all but 
finitely many primes p. Given such a $p$, we have $\exp(\mathcal G(0))=
\exp(\mathcal G(p))$ if $p$ does not divide the denominator 
of any coefficient of any element of $A$ occurring when testing whether $\mathcal G(0)$ 
is a Gr\"{o}bner basis using Buchberger's criterion. The result follows.
\end{proof}

\begin{lemma}
\label{lemma:L2-cond}
If $p$ is a prime satisfying condition \emph{(L2)}, then $\mathcal G(0)_p=\mathcal G(p)$.
\end{lemma}

\begin{proof}
We proceed as in the commutative setting. First of all, the graded case can be handled as in
\cite[Theorems 5.12 and 6.2]{A}. Next, we reduce the general case to the graded case by
adapting the proof of \cite[Theorem 2.4]{pfister2011}. For this, let $F(0)$ and $F(p)$ 
be finite sets of standard monomials obtained by applying Corollary \ref{Finite_det} to 
the Gr\"obner bases $\mathcal G(0)$ and $\mathcal G(p)$, respectively. Then apply 
\cite[Lemma 1.2.11]{GP7} to the set
\[
F = F(0)\cup F(p)\cup \{x_ix_j\mid i<j\}\cup \{x^{\alpha}\mid \alpha \in \mathcal N(d_{ij}) \text{ for some } d_{ij}\}
\]
to get a vector $\omega\in\N^n_{\geq 1}$ such that, for all $x^\alpha, x^\beta\in F$, we have 
$x^\alpha > x^\beta$ iff $x^\alpha >_\omega x^\beta$. Then, in particular, $>_\omega$ is admissible for $A$. 
Our choice of $F(0)$ gives that $\exp_{>}(\mathcal G(0))=\exp_{>_\omega}(\mathcal G(0))$ 
and that $\mathcal G(0)\subset A$ is the reduced Gr\"obner basis for $I$ also with respect to $>_\omega$. 
Similarly, $\exp_{>}(\mathcal G(p))=\exp_{>_\omega}(\mathcal G(p))$ and $\mathcal G(p)\subset A_p$ 
is the reduced Gr\"obner basis for $\tilde{I}_p$  also with respect to $>_\omega$. Passing to the Rees  algebra
$R^{\;\!\omega} \! A$, it follows from Proposition \ref{GBHomogenization}{(1)} that $\mathcal G(0)^{h}$ and 
$\mathcal G(p)^{h}$ are the reduced Gr\"obner bases for  
$I^h$ and $(\tilde{I}_p)^h$,
respectively, with $\exp_{>_{{\omega}}^h}(\mathcal G(0)^{h})=\exp_{>_{{\omega}}^h}(\mathcal G(p)^{h})$. 
Since the result holds in the graded case, we conclude that $(\mathcal G(0)_p)^{h} = 
(\mathcal G(0)^{h})_p=\mathcal G(p)^{h}$ and, thus,  that $\mathcal G(0)_p=\mathcal G(p)$.
\end{proof}

Error tolerant rational reconstruction as introduced in~\cite{bdfp} makes use
of Gaussian reduction.  If applied as discussed in what follows,  the finitely
many primes not satisfying condition {(L1)} will not influence the 
final result. We start with a definition which reflects that we rely on Gaussian reduction:

\begin{definition}[\cite{bdfp}]
If $\mathcal{P}$ is a finite set of primes, set
\[
N^{\prime}=\prod_{p\in\mathcal{P}\text{ lucky}}p \hspace{0.5cm} \text{and}%
\hspace{0.5cm} M=\prod_{p\in\mathcal{P}\text{ unlucky}}p\text{.}
\]
Then $\mathcal{P}$ is called \emph{sufficiently large} if
%\begin{align*}%
\[
N^{\prime}> (a^{2}+b^{2}) \cdot M
\]
for any coefficient $\frac{a}{b}$ of any element of $\mathcal G(0)$ (assume
$\mathop{gcd}(a,b)=1$).
\end{definition}

\begin{lemma}
\label{lemma:lifting-correct}
If $\mathcal{P}$ is a sufficiently large set of primes satisfying condition \emph{(L2)}, 
then the the reduced Gr\"obner bases $\mathcal G(p)$, $p\in\mathcal{P}$, lift via Chinese remaindering
and error tolerant rational reconstruction to the reduced Gr\"obner basis $\mathcal G(0)$. 
\end{lemma}
\begin{proof} 
Since all primes in  $\mathcal{P}$ satisfy condition {(L2)}, Lemma \ref{lemma:L2-cond}  
gives $\mathcal G(0)_p=\mathcal G(p)$ for each $p\in\mathcal{P}$. 
Since $\mathcal{P}$ is sufficiently large, the result follows as in the proof of 
\cite[Lemma 5.6]{bdfp} from \cite[Lemma 4.3]{bdfp}.
\end{proof}

Lemma \ref{lem finite} guarantees, in particular, that a sufficiently large
set $\mathcal{P}$ of primes satisfying condition {(L2)} exists. So 
from a theoretical point of view, the idea of finding $\mathcal G(0)$ is 
now as follows: Consider such a set $\mathcal{P}$, compute the reduced 
Gr\"{o}bner bases $\mathcal G(p)$, $p\in\mathcal{P}$, and lift the results 
to $\mathcal G(0)$.

>From a practical point of view, however, we face the problem that condition (L2) can only 
be checked a posteriori. On the other hand, as already pointed out, we need that the $\mathcal G(p)$, 
$p\in \mathcal{P}$, have the same set of leading monomials in order to identify corresponding 
Gr\"{o}bner basis elements in the lifting process. To remedy this situation, we suggest to 
proceed in a randomized way: First, fix an integer $t\geq1$ and choose a set of $t$ primes
$\mathcal{P}$ at random. Second, compute  $\mathcal{GP}=\{\mathcal G(p)\mid p\in\mathcal{P}\}$,
and use a majority vote:

\vspace{0.2cm}

\noindent\emph{\textsc{deleteByMajorityVote:} Define an equivalence relation
on $\mathcal{P}$ by setting $p\sim q:\Longleftrightarrow\exp(\mathcal G(p))=\exp(\mathcal G(q))$.
Then replace $\mathcal{P}$ by an equivalence class of largest
cardinality\footnote{We have to use a weighted cardinality
count: when enlarging $\mathcal{P}$, the total weight of the elements already
present must be strictly smaller than the total weight of the new elements.
Otherwise, though highly unlikely in practical terms, it may happen that only
unlucky primes are accumulated.}, and change $\mathcal{GP}$ accordingly.}

\vspace{0.2cm}
\noindent Now, all $\mathcal G(p)$, $p\in\mathcal{P}$, have the same set of leading
monomials. Hence, we can apply the error tolerant lifting algorithm to the
coefficients of the Gr\"{o}bner bases in $\mathcal{GP}$. If this algorithm
returns \texttt{false} at some point, we enlarge the set $\mathcal{P}$ by $t$
primes not used so far, and repeat the whole process. Otherwise, the lifting
yields a set of elements $\mathcal G\subset A$. Furthermore, if $\mathcal{P}$ is
sufficiently large, all primes in $\mathcal{P}$ satisfy condition {(L2)}. Since
we cannot check, however, whether $\mathcal{P}$ is sufficiently large, 
we include a final (partial) verification step in characteristic zero as discussed below. 
Since this test is particularly expensive if $\mathcal G\neq \mathcal G(0)$, we first perform 
a test in positive characteristic in order to increase our chances that the two sets
are equal:

\vspace{0.2cm}
\emph{\textsc{pTest:} Randomly choose a prime $p\notin\mathcal{P}$ which does
neither divide the numerator nor the denominator of any coefficient occurring in any
element of $\mathcal G$. Return \texttt{true} if $\mathcal G_{p}=
\mathcal G(p)$, and \texttt{false} otherwise.}

\vspace{0.2cm}
\noindent
If \textsc{pTest} returns \texttt{false}, then $\mathcal{P}$ is not
sufficiently large (or the extra prime chosen in \textsc{pTest} is unlucky). 
In this case, we enlarge $\mathcal{P}$ as above and repeat the process. If
\textsc{pTest} returns \texttt{true}, however, then most likely 
$\mathcal G=\mathcal G(0)$. In this case, we verify at least
that $\mathcal G$ is a left Gr\"obner basis, and that the left ideal 
$\langle \mathcal G\rangle$ generated by $\mathcal G$
contains the given left ideal $I$ (in the graded case discussed 
below, these two conditions actually guarantee that 
$\langle \mathcal G\rangle =I$). If the (partial) verification 
fails, we again enlarge $\mathcal{P}$ and repeat the process. 
We summarize this approach in Algorithm \ref{modGalgo} 
(as before, we ignore the primes which do not fulfil condition 
\eqref{cond:star}).

\begin{algorithm}
\caption{Modular Gr\"obner Basis Algorithm}\label{modGalgo}
\begin{algorithmic}[1]
\REQUIRE A nonzero left ideal $I\subset A$ given by finitely many generators, and
an admissible monomial ordering for $A$.
\ENSURE A subset $\mathcal G\subset A$ which is expected to be a Gr\"obner basis for $I$
               \phantom{AA}\quad\;(in the graded case, $\mathcal G$ is guaranteed to be such a Gr\"obner basis).
\STATE choose a set $\mathcal P$ of random primes
\STATE $\mathcal{GP}=\emptyset$
\LOOP
\FOR{$ p \in \mathcal P$}
\STATE compute $\mathcal{G}(p)\subset A_{p}$
%\STATE compute Gr\"obner bases $\mathcal{G}_p$ of $\langle (f_1)_p,\ldots, (f_t)_p \rangle\subset A_p$ 
\STATE $\mathcal{GP}=\mathcal{GP}\cup \mathcal{G}(p)$
\ENDFOR
\STATE $(\mathcal{P},\mathcal{GP})=\text{\emph{\textsc{deleteByMajorityVote}}}(\mathcal{P},\mathcal{GP})$
%\STATE ($\mathcal{GP},\mathcal{P}$)$:=$ \textit{Delete by majority vote} ($\mathcal{GP},\mathcal{P}$)
\STATE lift the Gr\"obner bases in $\mathcal{GP}$ to $\mathcal G\subset A$ via Chinese remaindering and error tolerant rational reconstruction
\IF {the lifting succeeds and \emph{\textsc{pTest}}$(I, \mathcal G,\mathcal{P})$}
%\STATE $\mathcal{G}_N:=$lift ($\mathcal{GP},\mathcal{P}$) to $A_N$, by applying CRT, where $N=\prod_{p\in \mathcal{P}}p$
%\STATE $\tilde{\mathcal{G}}:=$lift $\mathcal{G}_N$ to $A$ by rational reconstruction 
 %  	\IF{ $pTest(I,\mathcal{G}, \mathcal P)$}
\IF{ $\mathcal{G}$ is a Gr\"obner basis for $\langle \mathcal{G}\rangle$}
\IF{ $I\subset \langle \mathcal{G}\rangle$} 
\RETURN $\mathcal{G}$
\ENDIF
\ENDIF
\ENDIF
\STATE enlarge $\mathcal{P}$ with primes not used so far
\ENDLOOP
\end{algorithmic}
\end{algorithm}

Now, we address the graded case. We suppose that there is an $\omega\in\N_{\geq 1}^n$ such that $A$ and 
$I\subset A$ are graded with respect to the $\omega$-weighted degree as in Subsection
\ref{subsec: gradings}, and that $I$ is given by $\omega$-homogeneous generators. 
We use the index $d$ to indicate the graded 
pieces of $A$ and $I$ of degree $d$. Similarly for the other rings and ideals
considered in this section (such as $A_0$, $I_0$, and the $\tilde{I}_p$), 
which all inherit the grading.

We proceed by  considering Hilbert functions as in Arnold's 
paper ~\cite{A} which handles the commutative case.

\begin{lemma}\label{hfiip}
With notation and assumptions as above, let $p$ be a prime. Then 
$ H_{I_p}(d)  \leq H_{I}(d)$ for each $d \in \mathbb{N}$. 
\end{lemma}
\begin{proof} 
Fix a degree $d \in \mathbb{N}$.  We must show that 
$\dim_{\ \!\! \mathbb{Z}/p\mathbb{Z}}\;\! (I_p)_d \leq \text{dim}_{\ \!\!\Q} \;\! I_d$. For 
this, first note that $(I_0)_d$ is a free $\mathbb{Z}$-submodule of $(A_0)_d$
of finite rank. Let $\mathcal B =\{b_1,\ldots,b_m\}$ be a $\mathbb{Z}$\;\!-basis for $(I_0)_d$. 
Then note that for each $f\in I$, there is an integer $a \in \Z$ such that
$a\cdot f\in I_0$. This implies that $\mathcal B$ is also a $\Q$-basis for $I_d$, so that
$
\text{dim}_{\ \!\!\Q}\;\! I_d = \text{rank}_{\ \!\! \mathbb{Z}}\;\!(I_0)_d.
$
Furthermore, the reduction $\mathcal B_p$ still generates the 
$\mathbb{Z}/p\mathbb{Z}$-vector space $I_p(d)$. 
Hence,
$
\text{dim}_{\ \!\! \mathbb{Z}/p\mathbb{Z}}\;\!(I_{p})_d\leq\text{rank}_{\ \!\! \mathbb{Z}}\;\!(I_0)_d.
$
We conclude that 
$
\text{dim}_{\ \!\! \mathbb{Z}/p\mathbb{Z}}\;\! (I_p)_d \leq \text{dim}_{\ \!\!\Q} \;\! I_d,
$
as claimed.
\end{proof}

We can now prove:

\begin{theorem}[Final Verification, Graded Case]\label{hom} With notation and assumptions
as above, suppose that
\begin{enumerate}
\item\label{thm-two} $\exp(\mathcal G)=\exp(\mathcal G(p))$ for some prime $p$,
\item\label{thm-one} $\mathcal G$ is a Gr\"obner basis, and 
\item\label{thm-three} $I\subset \langle \mathcal G\rangle$.
\end{enumerate} 
Then  $\mathcal G$ is a Gr\"obner basis for $I$.
\end{theorem}
\begin{proof}
The result will follow from the second assumption once we show that  
$I=\langle \mathcal G \rangle$. Since $I\subset \langle \mathcal G \rangle$ by 
the third assumption, it suffices to show that $H_{I}(d) = H_{\langle \mathcal G \rangle}(d)$ 
for all $d\in\N$. This, in turn,  holds since we have
$$
H_{I}(d) \leq H_{\langle \mathcal G \rangle}(d)   = H_{\tilde{I}_p}(d) \leq H_{{I}_p}(d)\leq H_{I}(d)
$$
for each $d$ and each prime $p$ satisfying the first assumption. Indeed, the first and second
inequality are clear since $I\subset \langle \mathcal G \rangle$
and $\tilde{I}_p\subset I_p$, respectively; the equality follows from the first assumption (see Remark 
\ref{rem:Macaulay}); the third inequality has been established in Lemma \ref{hfiip}.
\end{proof}

\begin{remark}
Note that in all non-graded examples where we could check the output of
Algorithm \ref{modGalgo} by computing the desired Gr\"obner basis also 
directly over $\Q$, the result was indeed correct.
\end{remark}

\section{Timings}\label{compresult}
We have implemented our modular algorithm for computing Gr\"obner bases
in $G$-algebras over $\mathbb{Q}$ in the subsystem \plural \cite{plural, VikHans} 
of the computer algebra system \singular~\cite{singular}. This system offers
two variants of Buchberger's algorithm which within \plural  are adapted to the 
noncommutative case: While the \std command refers to the default version 
of Buchberger's algorithm in \singular, the ideas behind \slimgb aim at keeping 
elements short with small coefficients. 

In this section,  we compare the performance of the modular algorithm with that of 
\std and \slimgb applied directly over the rationals. In the tables below, when 
referring to the modular algorithm, we write modular \std \
respectively  modular \slimgb to indicate which version of 
Buchberger's algorithm is used for the mod $p$ computations.

We have carried out the computations on a Dell PowerEdge R720 with two 
Intel(R) Xeon(R) CPU E5-2690  @ 2.90GHz, 20 MB Cache, 16 Cores, 32 Threads, 
192 GB RAM with a Linux operating system (Gentoo).

In the tables, we abbreviate seconds, minutes, hours as s,m,h and
threads as thr. The symbol $\infty$ indicates that the computation 
did not finish within $25$ days or was halted since it consumed
more than 100 GB of memory.

\subsection{Examples From $D$-Module Theory Involving the Weyl Algebra}
\label{subsect:BSP}

We consider families of ideals which are computationally challenging and of interest 
in the context of $D$-modules,  specifically in the context of Bernstein-Sato 
polynomials. 

\subsubsection{The Setup}
\label{subsubsect:set-up}

Let $K$ be  a field of  characteristic zero, and consider a
non-constant polynomial $f\in K[x_1,\ldots,x_n]$. Write
$$
D_n(\field)=\field\langle x_1,\ldots, x_n,
\partial _1,\ldots \partial _n \mid \partial_i x_i=x_i\partial _i +1, \partial _i x_j=x_j \partial _i \ \text { for }\ i\neq j\rangle, 
$$
for the $n$-th  Weyl algebra as in Example \ref{ex1},  let $s$ be an extra variable, and set 
$K[x]_f[s]=K[x_1\dots,x_n]_f\otimes_{K} K[s]$ and $D_n(K)[s] = D_n(K) \otimes_{K} K[s]$. 
Let $K[x]_f[s]f^s$ stand for the free $K[x]_f[s]$-module of rank one generated 
by the symbol $f^s$. This is a left $D_n(K)[s]$-module with the action of a vector 
field $\theta$ being defined by the formula $\theta\cdot f^s=\theta(f^s)=s 
\theta(f) f^{-1} f^s$ and the product rule.  Consider the left annihilator $\Ann_{D_n(K)[s]}(f^s)
\subset D_n(K)[s]$. The Bernstein-Sato polynomial $b_f\in K[s]$ is the nonzero monic 
polynomial of smallest degree such that there exists an operator $P\in D_n(K)[s]$ with 
%$P(f^{s+1})=b(s)f^s$.
\begin{equation}\label{bseq1}
b_f - P\cdot f \in \Ann_{D_n(K)[s]}(f^s).
\end{equation} 
Put differently, $b_f$ is defined to be the monic generator of the ideal
$$
 \left(\Ann_{D_n(K)[s]}(f^s)+ {}_{D_n(K)[s]}\langle f\rangle \right) \cap K[s]
$$
which, by a result of Bernstein \cite{bs},  is nonzero. More generally, given 
polynomials $f_1,\dots, f_r\in K[x_1,\ldots,x_n]$ and extra variables $s=s_1,\dots, s_r$, 
consider the symbol $f^s=f_1^{s_1}\cdots f_r^{s_r}$. Then the analogous construction 
yields the \emph{Bernstein-Sato ideal}
$$
\mathcal{B}_{f_1,\ldots,f_r}(s) = 
\left(
\Ann_{D_n(K)[s]}(f^{s})+ {}_{D_n(K)[s]}\langle f_1 \cdots f_r\rangle\right) \cap K[s],
$$
which is nonzero by a result of Sabbah \cite{Sabbah87a}.  

\subsubsection{Computing the Annihilator}
\label{subsect:ann}

There are several algorithms for computing $\Ann_{D_n(K)[s]}(f^s)$ (see \cite{ABLMS}). For our tests here,
we use the method of Brian\c{c}on and Maisonobe which can be described as follows:
Consider the $r$th shift algebra 
$$
S_r(K)=\langle s_1.\dots, s_r, t_1, \dots, t_r \mid t_j s_k = s_k t_j - \delta_{jk} t_j  \rangle
$$
as in Example \ref{ex:ex2}, the tensor product
$$
A=D_n(K)\otimes S_r(K),
$$
and the left ideal 
\[
I =\left\langle s_j +  f_j t_j, \sum^r_{k=1} \frac{\partial f_k}{\partial x_i} {t}_k + \partial_i
\mathrel{\bigg|} 1\leq j \leq r, \ 1\leq i \leq n\right\rangle\  
\subset A.
\]

Then $\Ann_{D_n(K)[s]}(f^s) = I \cap D_n(K)[s]$. Hence, the annihilator is obtained by computing
a left Gr\"obner basis for $I$ with respect to an elimination ordering for $t_1,\dots, t_r$.

\subsubsection{Computing the Bernstein-Sato Ideal}
\label{subsect:BS}

By its very definition, the Bern\-stein-Sato ideal and, thus, the Bern\-stein-Sato polynomial
if $r=1$
can be found by computing a left Gr\"obner basis for 
$$\Ann_{D_n(K)[s]}(f^s)+{}_{D_n(K)[s]}\langle f_1\cdots f_r\rangle$$ with respect to an elimation ordering 
for $x_1,\ldots,x_n$, $\d_1,\ldots,\d_n$.

\begin{remark} There are more effective ways of computing Bernstein-Sato polynomials.
The method described above, however, allows one to compute Bernstein-Sato
ideals in general. See \cite{ABLMS} for more details.
\end{remark}

\subsubsection{Explicit Examples}

We focus on the computation of Bernstein-Sato polynomials
as outlined above (the case $r=1$). In all examples presented in
what follows, the time for computing the annihilator in 
\ref{subsect:ann} is negligible.  We will therefore only list 
the time needed for the elimination step in
\ref{subsect:BS}. Here, we use the block ordering
obtained by composing the respective degree reverse lexicographical 
orderings.

\begin{example}[$\text{Reiffen}(p,q)$, \cite{Reiffen}]\label{ReiffenEx}
We consider the family of polynomials 
\[
x^p + y^q + xy^{q-1} \in  \mathbb{Q}[x,y],\text{ where }q\geq p+1,
\]
and, correspondingly,  the second Weyl algebra $D_2(\mathbb{Q})$.

\vskip0.4cm
\begin{center}
\emph{
\begin{tabular}{|c|c|c|*{5}{c|}}
\hline
& \std & \slimgb &\multicolumn{5}{|c|}{modular \slimgb}\\ \cline{4-8} 
 & & & 1 thr  & 2 thr  & 4 thr & 8 thr & 16 thr \\
\hline
\hline
Reiffen(5,6) & $\infty$&63.86 h& 12.25 m & 7.21 m &4.7 m& 3.45 m&2.6 m \\
Reiffen(6,7) & $\infty$&$\infty$ &10.43 h & 6.03 h &4.65 h& 4.24 h &3.54 h\\
Reiffen(7,8) &$\infty$&$\infty$ &336.25 h &212.24 h & 170 h& 146 h&118 h\\ \hline
\end{tabular}
}
\end{center}

\vskip0.4cm
\noindent
For more insight, we also give timings for running our algorithm without 
the final tests  which check whether $\mathcal G$ is a left Gr\"{o}bner basis
and whether $I\subset \langle \mathcal G \rangle $
(see the discussion in Section \ref{sec:modgb}). We use just one thread.

\vskip0.4cm
\begin{center}
\emph{
\begin{tabular}{|c|c|c|}
\hline
& modular \slimgb & modular \slimgb without final tests \\  
\hline
\hline
Reiffen(5,6) & 12.25 m& 10.15 m \\
Reiffen(6,7) & 10.43 h &6.50 h \\
Reiffen(7,8) & 336.25 h &200.88 h \\ \hline
\end{tabular}
%\end{table}
}
\end{center}

\vskip0.4cm
\noindent
We see that for $\Reiffen(5,6)$, $\Reiffen(6,7)$, and $\Reiffen(7,8)$, the final tests take 
about 17\%, 37\%, and 40\% of the total computing time,~respectively.

\end{example}
\begin{example}\label{sim_Reiffen}
We consider the following polynomials with rational coefficients,
%\begin{center}
%\begin{tabular}{c}
\begin{itemize}
\item[]$f=xy^5z+y^6+x^5z+x^4y$,
\item[]$g= xy^6z+y^7+x^6z+x^5y$,
\item[]$h = (x-z)xyz(-x+y)(y+z)$,
\item[]$\cusp(p,q) = x^p-y^q,\text{ where }\gcd(p,q)=1$,
\end{itemize}
and, correspondingly,  the third and second Weyl algebras over
$\mathbb{Q}$,
respectively.
%\end{tabular}
%\end{center}
\noindent

\vskip0.4cm
\begin{center}
\emph{
\begin{tabular}{|c|c|c|c|}
\hline
&  \std& \slimgb & modular \slimgb  \\ 
\hline
\hline
f&$\infty$& 3.93 h & 3.59 h\\
g&  $\infty$ &$\infty$& 284.46 h  \\
h& $\infty$ &$\infty$ & 19.19 h  \\ 
\hline
cusp(9,8)&$\infty $ & 2.00 s& 30.81 s\\
cusp(10,9) &  $\infty$  &4.53 h& 3.17 h  \\
cusp(11,7)&$\infty $ &2.06 s& 2.18 m\\
cusp(11,8) & $\infty$ &3.17 h & 1.97 h  \\
cusp(12,7)&$\infty $ &9.53 s& 1.04 m\\
cusp(13,7)&$\infty $ &1.21 h& 40.32 m\\
\hline
\end{tabular}
}
\end{center}

\vskip0.4cm
\noindent
We observe that for the smaller examples such as  $\cusp(9,8)$, $\cusp(11,7)$, and  $\cusp(12,7)$, 
the \slimgb version of Buchberger's algorithm is superior due to the overhead of the modular algorithm.
\end{example}

Considering the substitution homomorphism $D_n(K)[s] \rightarrow D_n(K),\ s\mapsto -1$, 
it easily follows from Equation \ref{bseq1} that the Bernstein-Sato polynomial
${b_{f}(s)}$ is divisible by $s+1$ (recall that we suppose that $f$ is
non-constant). The polynomial $\frac{b_{f}(s)}{s+1} \in \field [s]$ 
is sometimes called the {\em reduced Bernstein-Sato polynomial}. It is 
easy to see that the following holds:
\begin{equation}
\label{equ:red-BSP}
\left\langle \frac{b_{f}(s)}{s+1}\right\rangle= \left (\Ann_{D_n(\mathbb{Q})[s]}(f^s)%+\langle f\rangle_{D_n(\mathbb{Q})[s]} +
+{\scriptsize{\phantom{}_{\raisebox{-2.5ex}{${D_n(\mathbb{Q})[s]}$}}}}\!\!\left\langle f, \frac{\partial f}{\partial x_{1}},
\ldots,\frac{\partial f}{\partial x_{n}}\right\rangle\right) \cap \field [s].
\end{equation}
Computing the Bernstein-Sato polynomial via this equation 
may be considerably faster than using the method described earlier: 
Compare the timings for $\cusp(13,7)$ in the tables above and below.

\begin{example} Equation \eqref{equ:red-BSP} allows us to compute the
Bernstein-Sato polynomials in some of the more involved  $\cusp(p,q)$
instances:

\vskip0.4cm
\begin{center}
%\begin{table}[h!]
%\caption{Timings using Equation \eqref{equ:red-BSP}}
\emph{
\begin{tabular}{|c|c|c|c|}
\hline
& \std & \slimgb & modular \slimgb  \\ 
\hline
\hline
cusp(13,7)&$\infty$ &3.94 m &1.21 m \\
cusp(13,8)&$\infty$ &$\infty$  &2.21 m \\
cusp(13,9)&$\infty$ &$\infty$ & 5.67 m \\
cusp(13,10)&$\infty$ &$\infty$  & 9.43 m \\
cusp(13,11)&$\infty$ &$\infty$  & 18.71 m \\
cusp(13,12)&$\infty$&$\infty$ & 27.25 m \\
cusp(14,9)&$\infty$ &$\infty$ & 7.83 m \\
cusp(14,11)&$\infty$&$\infty$  & 27.08 m \\
cusp(14,13)&$\infty$&$\infty$ & 1.16 h \\
cusp(15,7)&$\infty$ &2.74 h & 2.15 m \\
cusp(15,8)&$\infty$ & $\infty$  & 4.00 m \\
cusp(15,11)&$\infty$& $\infty$  & 36.01 m \\
cusp(15,13)&$\infty$& $\infty$ & 1.56 h \\
cusp(17,13)&$\infty$& $\infty$  &3.23 h\\
cusp(19,13)&$\infty$& $\infty$ &6.12 h\\
cusp(19,17)&$\infty$& $\infty$ &29.06 h\\
\hline
\end{tabular}
}
%\end{table}
%\vspace{0.2cm}
%{\small{Timings using Equation \eqref{equ:red-BSP}}}
\end{center}

\vskip0.4cm
\end{example}

\begin{example}\label{jacobian_trickEx} We consider the polynomial 
$$
f=(x^4+y^4)(w^2+z^2)(x+z)\in\mathbb{Q}[w,x,y,z],
$$
and compute the Bernstein-Sato polynomial $b_f$ using  Equation \eqref{equ:red-BSP}:
 
\vskip0.4cm
\begin{center}
%\begin{table}[h!]
%\caption{Timings using Equation \eqref{equ:red-BSP}}
\emph{
\begin{tabular}{|c|c|c|*{5}{c|}}
\hline
& \std & \slimgb &\multicolumn{5}{|c|}{modular \slimgb }\\ \cline{4-8} 
 & & & 1 thr  & 2 thr  & 4 thr & 8 thr & 16 thr \\
\hline
\hline
f & $\infty$& $\infty$ &531.71  h& 322.68 h &205.85 h&118.30 h&88.01 h \\ \hline
\end{tabular}
}
%\end{table}
%\vspace{0.2cm}
%{\small{Timings using Equation \eqref{equ:red-BSP}}}
\end{center}

\vskip0.4cm
\noindent
For the polynomial 
$$
g=(x^5+y^5)(w^2+z^2)(x+z) \in\mathbb{Q}[w,x,y,z],
$$
already the Gr\"obner basis computation over $\mathbb{F}_p$, for just one
randomly selected \singular prime $p$, takes 240 hours. 
The direct computation  over $\mathbb{Q}$ using \std and \slimgb runs out of memory.
\end{example}

\subsection{Some Well-Known Benchmark Examples}

\begin{example}\label{qc_ex}
We consider the quasi-commutative  graded $\mathbb{Q}$-algebra
\[ A=\mathbb{Q}\langle x_1,\ldots,x_n \mid x_jx_i=2x_ix_j,~ 1 \leq i < j \leq n\rangle\]
together with the degree reverse lexicographic ordering (and, thus, the weight vector 
$\omega=(1,\dots,1)$). In the corresponding Rees algebra, we compute left Gr\"obner
bases for homogenized versions of the benchmark systems 
\emph{cyclic}(n), \emph{katsura}(n), \emph{reimer}(n), and \emph{eco}(n) (see \cite{sd}).
Here are the timings:

\vskip0.4cm
\begin{center}
\emph{
%\begin{table}[h!]
%	\centering
%\caption{Performance of modular algorithm versus slimgb:}
\begin{tabular}{|c|c|*{5}{c|}}
\hline
& \slimgb &\multicolumn{5}{|c|}{modular \slimgb }\\ \cline{3-7} 
 & & 1 thr  & 2 thr  & 4 thr & 8 thr & 16 thr \\
\hline
\hline
cyclic(7) & 11.24 m& 27.66 s& 16.13 s &9.66 s & 7.81 s &6.64 s\\
cyclic(8) &55.28 h &2.51 h & 1.21 h   &34.65 m & 27.64 m &17.13 m\\ \hline
katsura(9) &4.49 m &1.51 m & 49.27 s &30.60 s&21.77 s&16.28 s  \\
katsura(10) & 10.65 h& 26.83 m&14.54 m   &8.59 m& 3.53 m& 3.38 m\\
katsura(11) &199.71 h & 4.32 h& 2.76 h &1.59 h& 46.48 m&24.52 m  \\
katsura(12) &$\infty$ & 13.78 h& 7.68 h& 4.40 h&2.34 h &1.46 h\\
katsura(13) &$\infty$&50.14 h&32.33 h &17.74 h& 10.72 h&5.80 h\\
\hline
reimer(4) &14.62 s & 3.14 s& 2.69 s  & 1.99 s & 1.58 s & 1.48 s  \\
reimer(5) &29.07 h &2.59 h& 1.57 h &58.47 m &26.33 m & 18.04 m  \\
\hline
eco(15) &25.93 h &  9.40 h &5.77 h &3.54 h & 2.55 h  & 1.83 h  \\
\hline 
\end{tabular}
}
%\end{table}
\end{center}

\vskip0.4cm
\end{example}

\subsection{A Remark on the Number of Primes}

\begin{remark}
\label{rem:number-primes}
The efficiency of our algorithm depends, in particular, on the number
of modular Gr\"obner basis computations before the lifting and testing
steps. In our implementation, this  is the smallest multiple of the number 
of available threads which is greater than or equal to $20$. 
\end{remark}

\section{Conclusion}\label{sect:con}
In this paper, we have introduced modular techniques for the computation of
Gr\"obner bases in \PBW-algebras defined over $\mathbb{Q}$. On the theoretical side, 
we have shown that the final verification test for graded ideals, which 
is well-known from the commutative case, also works in the noncommutative 
setting. On the practical side, we have implemented our modular algorithm
in the subsystem \plural of \singular and have demonstrated that the new 
algorithm is typically superior to the non-modular versions of Buchberger's 
algorithm in \plural. 

\bibliographystyle{abbrv}
%\bibliography{myref}  
%VL: the contents of .bbl file are inserted here after the compilation

\end{document}